\documentclass[12pt]{article}
\usepackage{bbm}
\usepackage[english]{babel}
\usepackage[latin1]{inputenc}
\usepackage{booktabs}
\usepackage{graphicx}
\usepackage{graphics}
\usepackage{caption}
\usepackage{amsmath}

\usepackage{amssymb}
\usepackage{amsthm}
\theoremstyle{plain}
\newtheorem{thm}{Theorem}[section]

\newtheorem{prop}[thm]{Proposition}

\theoremstyle{definition}
\newtheorem{defn}{Definition}[section]
\theoremstyle{remark}
\newtheorem*{rem}{Remark}

\usepackage{hyperref}
\usepackage{array}
\usepackage{color}
\usepackage[]{tabularx}
\usepackage{float}

\title{RISK MEASURES AND CREDIT RISK UNDER THE BETA-KOTZ DISTRIBUTION}

\author{M. ANDREA ARIAS-SERNA, FRANCISCO J. CARO-LOPERA,\\
	 JEAN-MICHEL LOUBES}

\begin{document}

\maketitle

\begin{abstract}
	This paper considers the use for Value-at-Risk computations of the so-called Beta-Kotz distribution based on a general family of distributions including the classical Gaussian model. Actually, this work develops a new method for estimating the Value-at-Risk, the Conditional Value-at-Risk and the Economic Capital when the underlying risk factors follow a Beta-Kotz distribution. After estimating the parameters of the distribution of the loss random variable,  both analytical for some particular values of the parameters and numerical approaches are provided for computing these mentioned measures. The proposed risk measures are finally applied for quantifying the potential risk of economic losses in Credit Risk.
	
keywords: Value-at-Risk; Conditional Value-at-Risk; Economic Capital; Credit Risk; Gauss hypergeometric function; Beta-Kotz distribution.

\end{abstract}

\section{Introduction}	

There is a variety of risk measures based on loss distributions to quantify the different types of risk; examples include the variance, the Value-at-Risk and the Conditional Value-at-Risk, see  \cite{Guegan(2012)}, \cite{McNeil(2015)}, \cite{Wagalath (2018)}, there is few work regarding  the pdf schemes. These include the standard Normal distribution \cite{Jorion(2007)}, \cite{Alexander(2008)}, which does not account for fat-tails and is symmetric, the Student-t distribution \cite{Lin(2006)}, which is fat tailed but symmetric, and the generalized error distribution (GED), which is more flexible than the Student-t including both fat and thick tails. These risk measures are typically statistical quantities describing the conditional or unconditional loss distribution of the portfolio over some predetermined time horizon. Since it is in general very difficult, if not impossible, to obtain analytically tractable expressions for the distribution function of the random variable, usually these risk measures are estimates by generating random samples from its distribution and computing the corresponding empirical quantiles.

As assert \cite{Crouhy et al. (2000)} while it was legitimate to assume normality of the portfolio changes due to market risk, it is no longer the case of credit returns which are by nature highly skewed and fat-tailed. Indeed, there is limited upside to be expected from any improvement in credit quality, while there is substantial downside consecutive to downgrading and default. The percentile of the distribution cannot be any longer estimated from the mean and variance only. The calculation for measure the credit risk requires simulating the full loss distribution of the changes in portfolio value.

In this paper, a Beta-Kotz distribution to model the credit loss rate in event of default is considered. The Beta-Kotz distribution modelling a wide class of data with different shapes in close domains, the distribution can be strongly right-skewed or less skewed as the parameters approach each other, also the distributions would be left-skewed if the parameters values were switched \cite{Wang(2005)}. We developed methods for the Value-at-Risk, Conditional Value-at-Risk and Economic Capital calculations that exploit the properties of the unique zero of the Gauss hypergeometric function. Certain restrictions on the values of the parameters of the Beta-Kotz distribution that allow us to solve the problem analytically are considered, that is, values for which it is possible to obtain  analytical results on the zeros of the Gauss hypergeometric functions, then, we deal with the numerical calculation of the real zeros of the Gaussian hypergeometric function. There are several methods proposed in the literature for the calculation of zeros of hypergeometric functions, some proposed methods are: Newton method and method based on the division algorithm \cite{Dominici(2013)}, asymptotic estimates \cite{Srivastava(2011)} \cite{Duren(2001)}, and Matrix methods \cite{Ball(2000)}. We will be based on the method of Newton; this method has the advantage that it is easy to understand and has a complete implementation in the routines packages of the R software.  It also has internal routines for evaluating functions hypergeometric see \cite{Fox(2016)} and \cite{Kasper (2017)} and references them.  The mentioned risk measures on credit risk framework to the financial analysis of the risk and performance of a financial institution are applied. The framework first quantifies the probability distributions of the loss given default derived from credit risk, then the risk measures are calculated using the approaches developed throughout the paper.  We want to stress however that the quantitative modelling of credit risk is just a motivating example where the techniques discussed in our paper can be applied naturally. The results obtained can be applied much more widely in banking.

The rest of this paper is organized as follows. Section 2 provides some risk measures based on the portfolio loss distributions. Section 3 derives a generalized Beta distribution, under the elliptical model. Section 4 defines the Beta-Kotz Value-at-Risk and demonstrates that it is unique in the $(0,1)$ interval. Section 5 uses some classical analytical methods with a view to studying the behaviour of risk measures seen as the zeros of the Gauss Hypergeometric polynomials. Section 6 provides numerical solutions for the values of risk measures may be viewed as the zeros of the Gauss Hypergeometric function. Section 7 Provides a description of the method of moments and maximum likelihood for the estimation of the Beta-Kotz parameters. Section 8 presents the case study on the credit risk framework. Section 9 presents the demonstrations of the proposed theorems and propositions. Section 10 provides the concluding remarks.

\section{Preliminary}

This section aims to familiarize the reader with some risk measures based on the portfolio loss distributions, these are typically statistical quantities describing the conditional or unconditional loss distribution of the portfolio over some predetermined time horizon. Let us introduce some basic concepts and notations.

Given the a random variable $X$ in a probability space, we denote its distribution function (d.f.) by $F_X$, unless otherwise stated, $F_{X}^{-1}$ is the ordinary inverse of $F_{X}$. $X$ is called the loss, hence $E(X)$ is called the expected loss.

Probably, the most commonly used risk measure in the financial area is the Value-at-Risk, which is defined in \cite{Rockafellar(2002)} as

\begin{defn} \label{VaR}
	The Value-at-Risk $(VaR_{\alpha}(X))$ for a portfolio with loss variable $X$ at the confidence level $\alpha \in (0,1)$ is a real number such that
	\begin{equation}
	F_{X}(VaR_{\alpha}(X))=P(X\leq VaR_{\alpha}(X))=\alpha.
	\end{equation}
\end{defn}
The determination of the $VaR_{\alpha}(X)$ is completely embedded in the knowledge of their distribution risk factors. This remark entails that the estimation of the Value-at-Risk should be carried determining the parameters of the distribution function of risk factors which, model the changes in the underlying risk factors. 

If $\hat{X}=\frac{X-\mu}{\sigma}$ and having estimated the unknown parameters of the model, the $VaR_{\alpha}(X)$ for the $\alpha$-percentile of the assumed distribution can be calculated straightforward using the equation (Tang \& Shieh, 2006).
$$VaR_{\alpha}(X)=\mu+F_{\hat{X}}^{-1}(\alpha)\sigma.$$

In the literature, there are few works related to others location-scale distribution family, among them the following are known.

If the loss $X$ has a normal distribution, with mean $\mu$ and standard deviation $\sigma$, then the $VaR_{\alpha}(X)$ is calculated as $$VaR_{\alpha}(X)=\mu+\sigma \Phi^{-1}(\alpha),$$
where $\Phi^{-1}$ is the inverse of the Normal standar distribution. see\cite{Jorion(2007)}, and \cite{Alexander(2008)} and references therein, for more details\\

If the loss $X$ has a Student t distribution with $\nu$ degrees of freedom, then

$$VaR_{\alpha}(X)=\mu+\sigma t^{-1}_{\upsilon}(\alpha),$$
where $t_{\upsilon}$ denotes the distribution function of standard $t$, see \cite{McNeil(2015)}, \cite{Glasserman(2002)} and references therein.

If $\hat{X}$ has a Pearson IV distribution, then the $VaR_{\alpha}(X)$ is calculated as

$$VaR_{\alpha}(X)=\mu+\sigma F^{-1}_{Pearson IV}(\alpha),$$ where at $F^{-1}_{Pearson IV}(\alpha)$, the inverse of the cumulative Pearson IV distribution at the specific confidence level is understood, see \cite{Stavroyiannis(2012)} for more details.

Usually, the approaches described are under the assumption of constant volatility over time. Hence, it is possible to incorporate models describing non-constant volatilities. In practice, there are numerous ARCH and GARCH models that can be chosen from, see \cite{Manganelli(2004)}, \cite{Stavroyiannis(2012)}, \cite{Han(2014)}, \cite{Gabrielsen(2015)} references therein.

Although the Value-at-Risk so defined is able to reflect risk aversion, it lacks some desirable properties such as subadditivity, which is the mathematical statement of the response of risk concentration, a basic reality in risk management. Among other objections raised on $VaR_{\alpha}(X)$ we can also mention that it is unable to account for the consequences of the established threshold being surpassed and that, in general, it is not continuous on the parameter $\alpha$, \cite{Arias(2016)}.

A measure of risk, closely related to $VaR_{\alpha}(X)$ is the Conditional Value-at-Risk $(CVaR_{\alpha}(X))$, defined as the conditional expected value of the $(1-\alpha)-tail$, it is defined in \cite{Rockafellar(2000)} as follows.

\begin{defn}
The Conditional Value-at-Risk of the loss associated at confidence level $\alpha\in (0,1)$ is the mean of the $\alpha$-tail loss distribution
\begin{equation}\label{CVaR}
CVaR_{\alpha}(X)=\dfrac{1}{1-\alpha}\int_{\alpha}^{1}VaR_{\nu}(X)d\nu.
\end{equation}
\end{defn}
If the loss distribution is normal with mean $\mu$ and variance $\sigma^{2}$, then
$$CVaR_{\alpha}(X)=\ \mu+\sigma\dfrac{\phi(\Phi^{-1}(\alpha))}{1-\alpha},$$
where $\phi$ is the density of the standar normal distribution.

If the loss $X$ has a standard t distribution with $\nu$ degrees of freedom, then

$$CVaR_{\alpha}(X)=\dfrac{g_{\nu}(t^{-1}_{\nu}(\alpha))}{1-\alpha}\left(\dfrac{\nu+(t^{-1}_{\nu}(\alpha))^{2}}{\nu-1}\right),$$

where $g_{\nu}$ denotes the density of standard normal $t$, see \cite{McNeil(2015)}.

There is a strong connection between $CVaR_{\alpha}(X)$ and $VaR_{\alpha}(X)$, which is evident in the dominance of the $CVaR_{\alpha}(X)$,  see Corollary 7 in \cite{Rockafellar(2002)}. They affirm that the $CVaR_{\alpha}(X)$ dominates $VaR_{\alpha}(X)$: $CVaR_{\alpha}(X)\geq VaR_{\alpha}(X)$. Indeed, $CVaR_{\alpha}(X)>VaR_{\alpha}(X)$ unless there is no chance of a loss greater than $VaR_{\alpha}(X)$.

Another measure closely related to $VaR_{\alpha}(X)$ is the  Economic Capital $(EC_{\alpha}(X))$ defined  as

\begin{defn}
The Economic Capital $EC_{\alpha}(X)$ for a given confidence level $\alpha$ is defined as the Value-at-Risk  $VaR_{\alpha}(X)$ minus the expected loss $E(X)$
\begin{equation}\label{EC}
EC_{\alpha}(X)=VaR_{\alpha}(X)-E(X).
\end{equation}
\end{defn}

Notice that getting the Economic Capital for location-scale distribution is only $\sigma F_{\hat{X}}^{-1}(\alpha).$

\section{Elliptical Beta Distribution}

The clasical two-parameter probability density function of the Beta distribution with shape parameters  $a$ and $b$ is given by
\begin{equation}
\frac{\Gamma\left(a+b\right)}{\Gamma\left(a\right)\Gamma\left(b\right)}
x^{a-1}(1-x)^{b-1},\ \  0\leq x\leq 1,\ \  a>0, \ \ b>0
\end{equation}

The Beta distribution is a useful model for a number of phenomena governed by random variables restricted to finite length domains \cite{Wang(2005)}. It is one of the most used in statistics since it models a wide class of data with different shapes in closed domains. The distribution can be strongly right-skewed or less skewed as the parameters approach each other, also the distributions would be left-skewed if the parameters values were switched, as can see in Figure 1.

\begin{figure}
	\centering
	\includegraphics[width=0.7\linewidth]{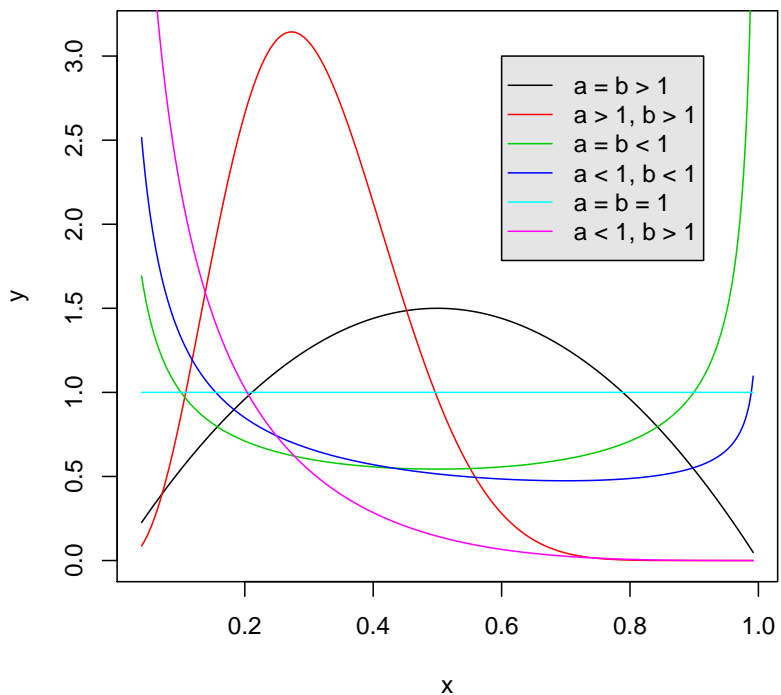}
	\caption{Shapes of Beta distribution}
	\label{fig:formasbetalegen}
\end{figure}

Various generalizations of the Beta distribution have been considered see, for instance, \cite{[Diaz(2008)}, \cite{Wang(2005)} and reference therein, but all these techniques are based on the hypothesis that some variables $A$ and $B$  are independent with Chi-Square distributions in the scalar case or Wishart distribution in the matrix case. In this section, we are generalizing these results, assuming the variables have an Elliptical Wishart distribution.

The generalized Elliptical Wishart distribution allows multiple extensions to robust Elliptical models when the Gaussian assumption is difficult to keep. The addressed distribution was derived by  \cite{Caro-Lopera(2014),Caro-Lopera(2014a)} and it is given as follows.

\begin{defn}
	The variable $X={Z}^{'}{Z}$, is said to have a Elliptical Wishart distribution with $n$ degrees of freedom $EW_{1}(n,\Sigma,h)$, if its p.d.f is given by
	\begin{equation}\label{ellipticalwishart}
	f_{X}({x})=	\frac{\pi^{\frac{n}{2}}}{\Gamma\left(\frac{n}{2}\right)\Gamma\left(\frac{n}{2}\right)
		\Sigma^{\frac{n}{2}}}x^{\frac{n}{2}-1}h_{1}\left(\Sigma^{-1}x\right),
	\end{equation}
	where $Z$ is a variable with Elliptically contoured $E(0, 1, h)$, with p.d.f given in \cite{Gupta(2013)} by
	$f_{Z}(Z)=\Sigma^{\frac{-n}{2}}h(\Sigma^{-1}ZZ^{'})$, where the function $h:\Re \rightarrow\ [0,\infty)$, called the generator function, is such that $\int_{0}^{\infty} u^{pn-1}h(u^2)du<\infty.$
\end{defn}
Now, we developed a result, analogous to those of \cite{Muirhead(2005)}, for introducing the Beta distribution, under elliptical models.

\begin{prop}	\label{1}

	Let $A$ and $B$ be independent univariate ellyptical Wishart distributions, where $A\sim EW_{1}(n_1,\Sigma,h_{1})$ and $B\sim EW_{1}(n_2,\Sigma,h_{2})$ with $\Sigma>0, n_1>0, n_2>0$. Put $A+B=T^2$ where $T>0$ and let $X$ be defined by $A=T^2X$, then  the density function of $X$ is
	\begin{equation}\label{betaelliptical}
	\frac{\pi^{\frac{n_{1}}{2}+\frac{n_{2}}{2}}}{\Gamma\left(\frac{n_{1}}{2}\right)\Gamma\left(\frac{n_{2}}{2}\right)}x^{-\frac{n_{2}}{2}-1}(1-x)^{\frac{n_{2}}{2}-1}		 \sum_{k=0}^{\infty}\frac{h_{2}^{(k)}(0)}{k!x^{k}(1-x)^{-k}}\int_{0}^{\infty}h_{1}(y)y^{\frac{n_{1}}{2}+\frac{n_{2}}{2}+k-1}dy.
	\end{equation}
	with $0\leq x\leq 1$.
	We will said that $X$ has an Elliptical Beta distribution indexed by the generators $h_{1}$ and $h_{2}$ with suitable existence conditions.
	
\end{prop}
\begin{proof}
	For the proof of the proposition see Appendix A.1
\end{proof}

The generalized Beta will be denoted by given $EB(n_{1},n_{2};h_{1},h_{2})$.

\begin{rem} The well known univariate Beta distribution is derived as follows: Take the gaussian generators 	 $h_{1}(y)=\frac{1}{(2\pi)^{\frac{n_{1}}{2}}}e^{-\frac{y}{2}}$,  $h_{2}(y)=\frac{1}{(2\pi)^{\frac{n_{2}}{2}}}e^{-\frac{y}{2}}$ and  $h_{2}^{(k)}(0)=\frac{(-1)^{k}}{(2\pi)^{\frac{n_{2}}{2}}2^{k}}$. Then, 	
	$$\int_{0}^{\infty}h_{1}(y)y^{\frac{n_{1}}{2}+\frac{n_{2}}{2}+k-1}dy=
	\frac{2^{\frac{n_{2}}{2}+k}}{\pi^{\frac{n_{1}}{2}}}\Gamma\left(\frac{n_{1}}{2}+\frac{n_{2}}{2}+k\right)$$
	For the summation, use $\Gamma(z+k)=\Gamma(z)(z)_{k}$.
	$$\sum_{k=0}^{\infty}\frac{(-1)^{k}(1-x)^{k}}{k!x^{k}}\Gamma\left(\frac{n_{1}}{2}+\frac{n_{2}}{2}+k\right)= 	 \Gamma\left(\frac{n_{1}}{2}+\frac{n_{2}}{2}\right)\sum_{k=0}^{\infty}\frac{\left(\frac{n_{1}}{2}+\frac{n_{2}}{2}\right)_{k}
	}{k!}\left(1-\frac{1}{x}\right)^{k}.$$
	Now, by zonal polynomial theory, $|\mathbf{I}_{m}-\mathbf{X}|^{-a}=\sum_{k=0}^{\infty}\sum_{\kappa}\frac{(a)_{\kappa}}{k!}C_{\kappa}(\mathbf{X})$, $||\mathbf{X}||<1$, see for example \cite{Muirhead(2005)}. Applying this for $m=1$, the summation just becomes $\left(\frac{1}{x}\right)^{-\left(\frac{n_{1}}{2}+\frac{n_{2}}{2}\right)}$. 
	
	Then the distribution $\frac{\Gamma\left(\frac{n_{1}}{2}+\frac{n_{2}}{2}\right)}{\Gamma\left(\frac{n_{1}}{2}\right)\Gamma\left(\frac{n_{2}}{2}\right)}x^{\frac{n_{1}}{2}-1}(1-x)^{\frac{n_{2}}{2}-1}$ is obtained.
\end{rem}
A number of generalized Beta distributions can be derived by using the following classical generators of elliptical distributions. As you can see, these generalizations are closely related to the hypergeometric Gauss function and the generalized hypergeometric Gauss function are defined in \cite{Baterman(1953)} and \cite{Caro2010} respectively.

\begin{defn} The Gauss hypergeometric function   $_2F_{1}(m,n;p;x)$ is defined as
\begin{equation}\label{hyper}
_2F_{1}(m,n;p;x)=\sum_{k=0}^{\infty}\dfrac{(m)_{k} (n)_{k} }{(p)_{k}k!}x^{k}.
\end{equation}
\end{defn}

Both the variable and the parameters of this equation can be complex, but in this work, we will assume, unless otherwise indicated, that both $x$ and $m,n$ and $p$ are real and $-p\notin N_0=\{0,1,2,...\}$,    and $ (m)_{k} = ( \Gamma(m+k))/( \Gamma(m))$ is Pochhammer\'s symbol, \cite{Baterman(1953)}.

This series converges when $|x|<1$ and when x = 1 provided that $Re(p - m - n) > 0$, and when $x = -1$ provided that $Re(p-m-n+1) > 0$.

\begin{defn} \label{ghyper} The generalized Gauss hypergeometric function   $_2P_{1}(f(k):m,n;p;x)$ is defined as
\begin{equation}
_2P_{1}(f(k):m,n;p;x)=\sum_{k=0}^{\infty}\dfrac{ (m)_{k} (n)_{k} }{(p)_{k}k!}f(k)x^{k}.
\end{equation}
\end{defn}

\subsection{Beta-Pearson model}

An interesting generalized Beta distribution can be derived from the Proposition  \ref{1} in terms of the Pearson VII model, in this case,

$$h(y)=\frac{\Gamma(s)}{(\pi R)^{\frac{n}{2}}\Gamma\left(s-\frac{n}{2}\right)}\left(1+\frac{y}{R}\right)^{-s}$$ and
$$h^{(k)}(0)=\frac{\Gamma(s)(-1)^{k}(s)_{k}}{R^{k}(\pi R)^{\frac{n}{2}}\Gamma\left(s-\frac{n}{2}\right)}.$$
Then $\int_{0}^{\infty}h_{1}(y)y^{\frac{n_{1}}{2}+\frac{n_{2}}{2}+k-1}dy=
\frac{R_{1}^{\frac{n_{2}}{2}+k}\Gamma\left(\frac{n_{1}}{2}+\frac{n_{2}}{2}+k\right)}{\pi^{\frac{n_{1}}{2}}}
\frac{\Gamma\left(s_{1}-\frac{n_{1}}{2}-\frac{n_{2}}{2}-k\right)}{\Gamma\left(s_{1}-\frac{n_{1}}{2}\right)}$ and the distribution  can be written as:

\begin{eqnarray*}
	&&
	\frac{\Gamma(s_{2})x^{-\frac{n_{2}}{2}-1}(1-x)^{\frac{n_{2}}{2}-1}}{\Gamma\left(\frac{n_{1}}{2}\right)\Gamma\left(\frac{n_{2}}{2}\right)
		\Gamma\left(s_{2}-\frac{n_{2}}{2}\right)\Gamma\left(s_{1}-\frac{n_{1}}{2}\right)}
	\\&&\times\sum_{k=0}^{\infty}
	\frac{(s_{2})_{k}\Gamma\left(\frac{n_{1}}{2}+\frac{n_{2}}{2}+k\right)\Gamma\left(s_{1}-\frac{n_{1}}{2}-\frac{n_{2}}{2}-k\right)}{
		(-1)^{k}
		k!x^{k}(1-x)^{-k}\left(\frac{R_{2}}{R_{1}}\right)^{\frac{n_{2}}{2}+k}}
\end{eqnarray*}
Now, using $\Gamma(z+k)=(z)_{k}\Gamma(z)$ and $\Gamma(z-k)=\frac{\Gamma(z)\Gamma(-z+1)}{(-1)^{k}\Gamma(-z+k+1)}$, we have:
\begin{eqnarray*}
	&&
	\frac{\Gamma\left(\frac{n_{1}}{2}+\frac{n_{2}}{2}\right)\Gamma(s_{2})\Gamma\left(s_{1}-\frac{n_{1}}{2}-\frac{n_{2}}{2}\right)
		x^{-\frac{n_{2}}{2}-1}(1-x)^{\frac{n_{2}}{2}-1}}
	{\Gamma\left(\frac{n_{1}}{2}\right)\Gamma\left(\frac{n_{2}}{2}\right)
		\Gamma\left(s_{1}-\frac{n_{1}}{2}\right)\Gamma\left(s_{2}-\frac{n_{2}}{2}\right)
	}
	\\&&\times\sum_{k=0}^{\infty}
	\frac{(s_{2})_{k}\left(\frac{n_{1}}{2}+\frac{n_{2}}{2}\right)_{k}}{
		k!x^{k}(1-x)^{-k}\left(\frac{R_{2}}{R_{1}}\right)^{\frac{n_{2}}{2}+k}\left(-s_{1}+\frac{n_{1}}{2}+\frac{n_{2}}{2}+1\right)_{k}}
\end{eqnarray*}

And using the generalized hypergeometric function (\ref{hyper}), we have:
\begin{eqnarray*}
	&&
	\frac{\Gamma\left(\frac{n_{1}}{2}+\frac{n_{2}}{2}\right)\Gamma(s_{2})\Gamma\left(s_{1}-\frac{n_{1}}{2}-\frac{n_{2}}{2}\right)
		x^{-\frac{n_{2}}{2}-1}(1-x)^{\frac{n_{2}}{2}-1}}
	{\Gamma\left(\frac{n_{1}}{2}\right)\Gamma\left(\frac{n_{2}}{2}\right)
		\Gamma\left(s_{1}-\frac{n_{1}}{2}\right)\Gamma\left(s_{2}-\frac{n_{2}}{2}\right)
	}
	\\\times &&{}_{2}P_{1}\left(\left(\frac{R_{1}}{R_{2}}\right)^{\frac{n_{2}}{2}+k}:s_{2},\frac{n_{1}}{2}+\frac{n_{2}}{2};
	-s_{1}+\frac{n_{1}}{2}+\frac{n_{2}}{2}+1;
	x^{-1}(1-x)\right)
\end{eqnarray*}

When $R_{1}=R_{2}$, the Beta-Pearson VII simplifies to:

\begin{eqnarray*}
	&&
	\frac{\Gamma\left(\frac{n_{1}}{2}+\frac{n_{2}}{2}\right)\Gamma(s_{2})\Gamma\left(s_{1}-\frac{n_{1}}{2}-\frac{n_{2}}{2}\right)
		x^{-\frac{n_{2}}{2}-1}(1-x)^{\frac{n_{2}}{2}-1}}
	{\Gamma\left(\frac{n_{1}}{2}\right)\Gamma\left(\frac{n_{2}}{2}\right)
		\Gamma\left(s_{1}-\frac{n_{1}}{2}\right)\Gamma\left(s_{2}-\frac{n_{2}}{2}\right)
	}
	\\\times &&{}_{2}F_{1}\left(s_{2},\frac{n_{1}}{2}+\frac{n_{2}}{2};
	-s_{1}+\frac{n_{1}}{2}+\frac{n_{2}}{2}+1;
	x^{-1}(1-x)\right)
\end{eqnarray*}

If $a=s_{2}, b=\frac{n_{1}}{2}+\frac{n_{2}}{2},
c=-s_{1}+\frac{n_{1}}{2}+\frac{n_{2}}{2}+1$, the Kummer relations of Erdelyi, p. 105, can provide a polynomial when $c-a$ or $c-b$ is a negative integer.

\subsection{Beta-Kotz model}

Note that, among any other restrictions, the Proposition \ref{1} can be applied if $h_{2}^{(k)}(0)$ exists and it is different from zero.  If we consider two Kotz type I models,
\begin{equation}
h(y)=\frac{r^{t_{1}+\frac{n_{2}}{2}-1}}{\pi^{\frac{n_{1}}{2}}\Gamma\left(t_{1}+\frac{n_{2}}{2}-1\right)}y^{t_{1}-1}e^{-ry}
\end{equation}

Then $h_{1}$ and $h_{2}$, depending on parameters $t_{i}\neq 1, s_{i}=1, r_{i}=R, i=1,2$, then  $h_{2}^{(k)}(0)=0$ and the distribution of the associated Beta must be derived by another method.  To get to this end, we just use the Taylor expansion in the joint density function of $T^2$ and $X$, so we have

\begin{equation*}
\frac{\pi^{\frac{n_{1}}{2}+\frac{n_{2}}{2}}x^{\frac{n_{1}}{2}-1}(1-x)^{\frac{n_{2}}{2}-1}(t^2)^{\frac{n_{1}}{2}+\frac{n_{2}}{2}-1}}{\Gamma\left(\frac{n_{1}}{2}\right)\Gamma\left(\frac{n_{2}}{2}\right)
	\Sigma^{\frac{n_{1}}{2}+\frac{n_{2}}{2}}}
h_{1}\left(\Sigma^{-1}t^{2}x\right)h_{2}\left(\Sigma^{-1}t^{2}(1-x)\right)dt^2dx.
\end{equation*}

Replacing the  generator of the Kotz function  $h_{i}(y)=\frac{R^{t_{i}+\frac{n_{i}}{2}-1}\Gamma\left(\frac{n_{i}}{2}\right)}{\pi^{\frac{n_{i}}{2}}
	\Gamma\left(t_{i}+\frac{n_{i}}{2}-1\right)}y^{t_{i}-1}e^{-Ry}$, $i=1,2$, we obtain
$$\dfrac{r^{t_1+\frac{n_1}{2}-1}r^{t_2+\frac{n_2}{2}-1}}{\Gamma(t_1+t_2+\frac{n_1}{2}+\frac{n_2}{2}-2)\Sigma^{(t_1+t_2+\frac{n_1}{2}+\frac{n_2}{2}-2)}}t^{2(t_1+t_2+\frac{n_1}{2}+\frac{n_2}{2}-3)}e^{-(r_1+r_2)\Sigma^{-1}t^2}\times $$
$$\dfrac{\Gamma(t_1+t_2+\frac{n_1}{2}+\frac{n_2}{2}-2)}{\Gamma(t_1+\frac{n_1}{2}-1)\Gamma(t_2+\frac{n_2}{2}-1)}x^{(t_1+\frac{n_1}{2}-1)-1}(1-x)^{(t_2+\frac{n_2}{2}-1)-1}dx dt^2$$

Then $T^2=A+B$ is independent of $X$, and the distribution of $X$ is 
\begin{equation}\label{U}
\dfrac{\Gamma(t_1+t_2+\frac{n_1}{2}+\frac{n_2}{2}-2)}{\Gamma(t_1+\frac{n_1}{2}-1)\Gamma(t_2+\frac{n_2}{2}-1)}x^{(t_1+\frac{n_1}{2}-1)-1}(1-x)^{(t_2+\frac{n_2}{2}-1)-1}
\end{equation}

Note that if $r_1=r_2=R$ then $T^2\sim EW_1(n_1+n_2, \Gamma,h(t_1+t_2-1,R,s=1)).$\\

Let us denote the Beta-Kotz distribution $KBeta(t_1+\frac{n_1}{2}-1, t_2+\frac{n_2}{2}-1)$ as the distribution of the random variable $X$ previously defined. Note also that if $t_1=t_2=1$, then $X\sim Beta(\frac{n_1}{2},\frac{n_2}{2})$.

\begin{rem}
	We check a corollary of the above general Beta-Kotz distribution. We can mix for example a Gaussian class model $h_{2}(y)=\frac{r^{\frac{n_{2}}{2}}}{\pi^{\frac{n_{2}}{2}}}e^{-ry}$ with exponential scale $r$  instead of $\frac{1}{2}$ and non vanishing $h_{2}^{(k)}(0)=\frac{(-1)^{k}r^{\frac{n_{2}}{2}+k}}{\pi^{\frac{n_{2}}{2}}}$,  with  a Kotz type I model, $h_{1}(y)=\frac{r^{t_{1}+\frac{n_{2}}{2}-1}}{\pi^{\frac{n_{1}}{2}}\Gamma\left(t_{1}+\frac{n_{2}}{2}-1\right)}y^{t_{1}-1}e^{-ry}$.
	After simplification the summation in (\ref{betaelliptical}) goes to\\ $\frac{\Gamma\left(\frac{n_{1}}{2}\right)\pi^{-\frac{n_{1}}{2}}-\frac{n_{1}}{2}}
	{\Gamma\left(t_{1}+\frac{n_{1}}{2}-1\right)}x^{-1+\frac{n_{1}}{2}+\frac{n_{2}}{2}+t_{1}}\Gamma\left(t_{1}+\frac{n_{1}}{2}+\frac{n_{2}}{2}-1\right)$, and then the associated Beta-Kotz is obtained as follows
	\begin{equation}
	\frac{\Gamma\left(t_{1}+\frac{n_{1}}{2}+\frac{n_{2}}{2}-1\right)}{\Gamma\left(\frac{n_{2}}{2}\right)\Gamma\left(t_{1}+\frac{n_{1}}{2}-1\right)}
	x^{t_{1}+\frac{n_{1}}{2}-2}(1-x)^{\frac{n_{2}}{2}-1}
	\end{equation}
	This coincides with (\ref{U}) when $t_{2}=1$, as we expect.
\end{rem}

\section{Value-at-Risk using the Beta-Kotz Distribution}\label{Valuerisk}

In this section we are studying Value-at-Risk, when the data come from the family of Beta-Kotz distributions introduced early.

\begin{thm} \label{t2}
Let	$X\sim KBeta(n_1,n_2,t_1,t_2)$ with $n_1>0, n_2>0$, $t_1+\frac{n_1}{2}-1>0$ and $t_2+\frac{n_2}{2}-1>0$, and, $F_X$ denote the cumulative distribution function. The Beta-Kotz Value-at-Risk of $X$ at probability level $\alpha\in(0,1)$ is obtained solving the hypergeometric equation:
\begin{equation}\label{w}
C\frac{VaR\beta_{\alpha}(X)^{t_1+\frac{n_1}{2}-1}}{(t_1+\frac{n_1}{2}-1)}{}_2F_{1}\left(t_1+\frac{n_1}{2}-1,-t_2-\frac{n_2}{2}+2;  t_1+\frac{n_1}{2},VaR\beta_{\alpha}(X)\right)-\alpha=0,
\end{equation}
where $C=\frac{\Gamma\left(t_1+t_2+\frac{n_1}{2}+\frac{n_2}{2}-2\right)}{\Gamma\left(t_1+\frac{n_1}{2}-1\right)\Gamma\left(t_2+\frac{n_2}{2}-1\right)}.$
\end{thm}

\begin{proof}
	For the proof of theorem see Apendix A.2.
\end{proof}
From now on, to simplify the notation,  let $a=t_1+n_1/2-1$ and $b=t_2+n_2/2-1$, then the equation (\ref{w}) is equivalent to

\begin{equation}\label{F1}
\dfrac{\Gamma(a+b)}{\Gamma(a)\Gamma(b)}\frac{VaR\beta_{\alpha}(X)^{a}}{a}{}_2F_{1}\left(a,1-b;a+1;VaR\beta_{\alpha}(X)\right)-\alpha=0
\end{equation}

Then the Beta-Kotz Value-at-Risk of $X$ at probability level $\alpha\in(0,1)$ is obtained solving the hypergeometric equation:
$$	\dfrac{\Gamma(a+b)}{\Gamma(a)\Gamma(b)}\sum_{k=0}^{\infty}\dfrac{(1-b)_{k}}{(a+k)k!} VaR\beta_{\alpha}(X))^{a+k}-\alpha=0$$

\begin{thm}\label{t3}
	The Beta-Kotz Value-at-Risk  is unique in the interval $(0,1)$.
\end{thm}

\begin{proof}
	For the proof of theorem see Apendix A.3.
\end{proof}

The Value-at-Risk for a Beta-Kotz distribution does not always have a closed expression which makes it practically impossible to obtain both the Conditional Value-at-Risk and the Economic capital defined by the equations \ref{CVaR} and \ref{EC}, however, as we will see in the next section depending on the values that take the parameters $a$ and $b$ it is possible to find exact analytical solutions, this also allows you to find exact solutions for the Value-at-Risk, the Conditional Value-at-Risk and for the  Economic capital.

\section{Analytical Estimation of Risk Measures}
In this section, we will be addressing certain restrictions on the values of parameters $a$ and $b$  that enable us to analyze the problem, that is, values for which it is possible to obtain families of variable changes that provide analytical results on the zeros of the Gauss hypergeometric functions.

If $m=-q$ and/or $n=-q$ is a negative integer, the series terminates and reduces to a polynomial of degre $n\in Z$ called hypergeometric polynomial of gradee $q$, see \cite{Baterman(1953)}. For instance, if $n=-q$, the hypergeometric function is reduced to the next polynomial
$$_2F_{1}(m,n;p;x)=\sum_{k=0}^{q}\dfrac{(m)_{k} (-q)_{k} }{(p)_{k}k!}x^{k}.$$

Then, if $1-b$ is a negative integer, the series, (\ref{w}), is reduced to a polynomial, and use the fact that $(-m)_{k}=(-1)^{k}\frac{\Gamma(m+1)}{\Gamma(m+1-k)}$, then by (\ref{F1}) to find the $VaR\beta_{\alpha}(X)$ is reduced to solve the following polynomial equation.

\begin{equation}\label{m}
\dfrac{\Gamma(a+b)}{\Gamma(a)}\sum_{k=0}^{b-1}\dfrac{(-1)^{k}}{k!(a+k)\Gamma(b-k)} (VaR\beta_{\alpha}(X))^{a+k}-\alpha=0.
\end{equation}

In the next propositions, we will deal with the analytical calculation of the $VaR\beta_{\alpha}(X)$ seen as real zeros of the hypergeometric polynomials (\ref{m}).

\begin{prop} \label{p1}
	\begin{enumerate}
		\item  The Beta-Kotz Value-at-Risk $VaR\beta_{\alpha}(X)$, for $b=2$, is determined by:
		\begin{itemize}
			\item If $a=b-1$ then
			\begin{equation}\label{V3}
			VaR\beta_{\alpha}(X)=1 -\sqrt{1-\alpha}
			\end{equation}
			
			\item  	If $a=b$ then
			\begin{equation}\label{V4}
			VaR\beta_{\alpha}(X)= \frac{1}{2}-\frac{1}{2}(-1+2\alpha+2\sqrt{\alpha^2-\alpha})^{\frac{1}{3}}-\frac{1}{2}(-1+2\alpha+2\sqrt{\alpha^2-\alpha})^{-\frac{1}{3}}
			\end{equation}	
			
			\item If $a=b+1$ then
			\begin{equation}\label{V5}
			VaR\beta_{\alpha}(X)=\frac{1}{6}\left[2+\sqrt{4+\frac{6(\alpha+Q^{2})}{R}}- \sqrt{2}\sqrt{4-\frac{3\alpha}{Q}}\right]+\frac{1}{6}\left[-3Q+\frac{4\sqrt{2}}{\sqrt{2+\frac{3(\alpha+R^{2})}{R}}}\right]
			\end{equation}
			
		\end{itemize}	
		
		\item The Beta-Kotz Value-at-Risk $VaR\beta_{\alpha}(X)$, for $b=3$, is determinted by:
		\begin{itemize}
			\item if $a=b-2$ then
			\begin{equation}\label{V6}
			VaR\beta_{\alpha}(X)=1+\sqrt[3]{\alpha-1}
			\end{equation}
			
			\item 	if $a=b-1$, the $VaR\beta_{\alpha}(X)$ is giben by
			
			{\footnotesize	\begin{equation}\label{V7} \frac{2}{3}+\frac{1}{2}\sqrt{\frac{8}{9}+\frac{16}{27\sqrt{\frac{4}{9}+\frac{2}{3}SP+\frac{2(1-\alpha)}{3SP}}}-\frac{2}{3}SP-\frac{2(1-\alpha)}{3SP}}-\frac{1}{2}\sqrt{\frac{4}{9}+\frac{2}{3}SP+\frac{2(1-\alpha)}{3SP}}
				\end{equation}}
		\end{itemize}
		
		\item The Beta-Kotz Value-at-Risk $VaR\beta_{\alpha}(X)$ for $b=4$ and $a=b-3$, is given by:
		\begin{equation}\label{V2}
		VaR\beta_{\alpha}(X)=1 -\sqrt[4]{1-\alpha}
		\end{equation}

		\item The Beta-Kotz Value-at-Risk $VaR\beta_{\alpha}(X)$ for $b=1$ and $a=b+n$,whit $n\in \textbf{N}_0$, is given by:
		
		\begin{equation}\label{V1}
		VaR\beta_{\alpha}(X)=\sqrt[n+1]{\alpha}.
		\end{equation}
		
	\end{enumerate}
\end{prop}

In particular note that when $a=b=1.$ the function $ Beta(a,b)=U(0,1)$, and the Beta-Kotz Value-at-Risk is given by
\begin{equation} \label{vu}
VaR\beta_{\alpha}(X)=\alpha.
\end{equation}

\begin{proof}
	For the proof of proposition see Apendix A.4.
\end{proof}

By Abel's theorem, there is no formula describing the roots of any general polynomial of degree $\geq 5$ of its coefficients and elementary operations. However, we can find the roots of both polynomials and the hypergeometric series using an algorithm as in the next section.

\begin{rem}
To find the $CVaR_{\alpha}(X)=\frac{1}{1-\alpha}\int_{\alpha}^{1} VaR_{\nu}(X)d\nu$ simply calculate the integral of the expressions from \ref{V3} to \ref{vu}. For instance, for the uniform typical case when $a=1$ and $b=1$, the $VaR_{\alpha}(X)=\alpha$ then
\begin{equation}
CVaR_{\alpha}(X)=\frac{1}{1-\alpha}\int_{\alpha}^{1}\nu d\nu =\dfrac{1+\alpha}{2}.
\end{equation}
\end{rem}

\begin{rem}
	To find the $EC_{\alpha}(X)=VaR_{\alpha}(X)-E(X)$ simply calculate the subtraction between  the mean of the Beta distribution $\dfrac{a}{a+b}$ and the expressions from \ref{V3} to \ref{vu} . For instance, for the uniform case
	\begin{equation}
	EC_{\alpha}(X)=\alpha-\dfrac{1}{2}.
	\end{equation}
	
\end{rem}

\begin{table}[!h]\label{tableVv}
	\caption{Analytical expressions for risk measures}

	\begin{center}
		\begin{tabular}{|c|c|c|c|c|}
			\hline 	$a$	&  $b$ &  $VaR\beta_{\alpha}(X)$  & $CVaR\beta_{\alpha}(X)$  & $EC\beta_{\alpha}(X)$
			\\
			\hline
				1& 1& $\alpha$	&$ \dfrac{1+\alpha}{2}$ & $\alpha-\dfrac{1}{2}$  \\
			\hline
			2& 1& $\sqrt[2]{\alpha}$&	$\dfrac{2(1-\sqrt{\alpha^{3}})}{3(1-\alpha)}$&	$\sqrt[2]{\alpha}-\dfrac{2}{3}$\\
			\hline
			3& 1& $\sqrt[3]{\alpha}$&$\dfrac{3(1-\sqrt[3]{\alpha^{4}})}{4(1-\alpha)}$&	$\sqrt[3]{\alpha}-\frac{3}{4}$\\
			\hline
			\dots& \dots&\dots &\dots	&\dots	 \\
			\hline
			n& 1& $\sqrt[n+1]{\alpha}$ & $\dfrac{(n+1)(1-\sqrt[n+1]{\alpha^{n+2}})}{(n+2)(1-\alpha)}$	& $\sqrt[n+1]{\alpha}-\dfrac{n+1}{n+2}$	 \\
			\hline
			1&2 &$1 -\sqrt{1-\alpha}$	&$1 -\dfrac{2}{3}\sqrt{1-\alpha}$&$\dfrac{2}{3} -\sqrt{1-\alpha}$\\
			\hline
			
			1 &3& $1+\sqrt[3]{\alpha-1}$ &$1+\dfrac{3\sqrt[3]{\alpha-1}}{4}$ & $\dfrac{3}{4}+\sqrt[3]{\alpha-1}$
			\\
			\hline
			1 &4&  $1 -\sqrt[4]{1-\alpha}$& $1-\dfrac{4\sqrt[4]{1-\alpha}}{5}$&$\dfrac{4}{5} -\sqrt[4]{1-\alpha}$
				\\
			\hline
		\end{tabular}

	\end{center}

\end{table}

Finally, in Table 2 numerical values of the three risk measures are provided, considering different values $a$ and $b$.

\begin{table}[!h]\label{tableVv}
	\caption{Risk Measures for different  $a$ and $b$}
		\begin{center}
		\begin{tabular}{|c|c|c|c|c|}
			\hline 	$a$	&  $b$ &  $VaR\beta_{0.99}(X)$  & $CVaR\beta_{0.99}(X)$  & $EC\beta_{0.99}(X)$
			\\
			\hline
		1& 1& 0.99	& 0.995 &	0.490  \\
		\hline
		2& 1& 0.995&	0.997&	0.325\\
		\hline
		3& 1& 0.997&	0.998&	0.248\\
		\hline
		
		1&2 &0.9	&0.933&	0.567\\
		\hline
		
		2&  2&  0.941&0.960& 0.442 \\
		\hline
		3&  2&0.958& 0.979& 0.357\\
		\hline

		1 &3&  0.785& 0.838& 0.534
		\\
		\hline
		
		2 &3&  0.859& 0.929& 0.460
		\\
		\hline
		
		1 &4&  0.684& 0.747& 0.484
			
		\\
		\hline

	\end{tabular}
	\end{center}
\end{table}

\section{Numerical Solutions}\label{numeric}

In this section we will deal with the numerical calculation of the three risk measures, all estimated from the calculation of the zeros of the hypergeometric function. We will be based on the method of Newton; this method has the advantage that it is easy to understand and has a complete implementation in the routines packages of the R software.  It also has internal routines for evaluating functions hypergeometric see \cite{Fox(2016)}, \cite{Kasper (2017)} and references therein.

The proposed algorithm to compute the $VaR\beta_{\alpha}$,  $CVaR\beta_{\alpha}$ and the $EC\beta_{\alpha}$,  is the following.\\

\textbf{The Algorithm}\\

library(hypergeo)\

library(rootSolve)\

poly:function(a,b,x,s) \{$\frac{\Gamma(a+b)}{\Gamma(a)\Gamma(b)}\frac{x^a}{a}hypergeo(a,1-b,a+1,x)-s$\}

$a\longleftarrow n$

$b\longleftarrow m$

$s\longleftarrow\alpha$

$d\longleftarrow0$

$k\longleftarrow0$

$x\longleftarrow seq(0,1)$

$y\longleftarrow rbeta(x,a,b)$

$fun \longleftarrow\ function (x) Re(poly(a,b,x,s))$

$VaR\beta_{\alpha} \longleftarrow\ uniroot.all(fun, c(0, 1))$

for (i in 1:length(x))\

\ \  \ \  \ \	\{if $(x[i]>VaRB)$\

\ \  \ \  \ \ \ \  \ \  \ \	\{$gg\longleftarrow gg+1$ 	

\ \  \ \  \ \ \ \  \ \  \ \ \ $k<-k+x[i]$\}
\}

$CVaRB\longleftarrow k/gg$

$EC\longleftarrow VaRB-mean(y)$
\\

Similarly to the previous subsection, in Table 3, we provide the Beta-Kotz risk measures. We have considered different values for the parameters $a$ and $b$, as well as we have considered different probability level $\alpha$. Note that the first nine columns marked in bold coincide with the results obtained in Table 1, which were obtained analytically.

\begin{table}[H]\label{tableVv}
	\caption{Risk Measures for different  $a$ and $b$}
	\begin{center}
		\begin{tabular}{|c|c|c|c|c|}
			\hline 	$a$	&  $b$ &  $VaR\beta_{0.99}(X)$  & $CVaR\beta_{0.99}(X)$  & $EC\beta_{0.99}(X)$
			\\
			\hline
				\textbf{1}  &  \textbf{1} & \textbf{0.990} &  \textbf{ 0.995}&  \textbf{0.490} 		
			\\
			\hline
			\textbf{2}  &  \textbf{1}&\textbf{ 0.995}& \textbf{0.9996}&  \textbf{0.328}
			\\
			\hline
			\textbf{3} &   \textbf{1}& \textbf{0.997}&  \textbf{0.996}&  \textbf{0.246}
			\\
			\hline
			
			\textbf{1} & \textbf{2}&  \textbf{ 0.900}& \textbf{0.957}&  \textbf{0.566}
			\\
			\hline
			
			\textbf{2}&  \textbf{2}&  \textbf{0.941}&  \textbf{0.971}&  \textbf{0.442}
			\\
			\hline

			\textbf{3}&  \textbf{2}&  \textbf{0.958}&  \textbf{0.979}&  \textbf{0.357}
			\\
			\hline

			\textbf{1} & \textbf{3}&  \textbf{0.785}&  \textbf{0.929} & \textbf{0.534}
			\\
			\hline

			\textbf{2}&    \textbf{3}& \textbf{0.859} & \textbf{0.929}&  \textbf{0.459}
			\\
			\hline

			\textbf{1} &   \textbf{4}& \textbf{0.684}&  \textbf{0.898}&  \textbf{0.484}
			\\
			\hline
			
			4.1   &1.0& 0.998& 0.999& 0.194\\
			\hline
			5.1 &  1.5& 0.989& 0.995& 0.216\\
			\hline
			4.1 &  4.1& 0.855& 0.933& 0.355\\
			\hline
			5.1 &  5.1& 0.827& 0.923& 0.327\\
			\hline
			6.0&   6.0& 0.806& 0.916& 0.307\\
			\hline
			0.6 &  0.6& 0.999& 0.916& 0.499\\
			\hline
			0.8 &  0.8& 0.996& 0.916& 0.495\\
			\hline
			1.2 & 11.4 &0.355& 0.785& 0.260\\
			\hline
			1.3 & 13.0& 0.330 &0.742 &0.239\\
			\hline
			1.5&  14.1& 0.327& 0.721& 0.231\\
			\hline
			2.0 & 19.0& 0.289& 0.704 &0.193\\
			\hline
			0.5 & 30.0& 0.106& 0.671 &0.090\\
			\hline	
					
		\end{tabular}
	\end{center}
\end{table}

\section{Parameter Estimations}\label{Betafit}

Common methods of estimation of the parameters of the Beta distribution are maximum likelihood and method of moments. The maximum likelihood equations for the Beta distribution have no closed-form solution; estimates may be found through the use of the iterative method, . The method of moments estimators have a closed-form solution. We examine both of these estimators here, see for instance  see \cite{Farnun}, \cite{Wang(2005)} and reference therein.

\subsection{\textbf{Maximum likelihood estimators}}
Assume that an iid sample $x_1,x_2,...,x_n$ of size $n$ has been collected for a random variable $X$ which follows Beta distribution \cite{Wang(2005)}.

$$B(a,b)=\dfrac{\Gamma(a+b)}{\Gamma(a)\Gamma(b)}x_i^{a-1}(1-x_i)^{b-1}$$

The logarithm of the likelihood function is given by
$$log L(a,b,x)=nlog \Gamma(a+b)-nlog\Gamma(a)-nlog\Gamma(b)+(a-1)\sum_{i=1}^{n}log(x_i)+(b-1)\sum_{i=1}^{n}log(1-x_i).
$$

By differentiating with respect to $a,b$ and equating to zero, the likelihood equations can be obtained.
$$ \frac{\partial log L}{\partial a}= n\frac{\Gamma^{'}(a+b)}{\Gamma(a+b)}-n\frac{\Gamma^{'}(a)}{\Gamma(a)}+\sum_{i=1}^{n}log(x_i)=0$$
$$ \frac{\partial log L}{\partial b}= n\frac{\Gamma^{'}(a+b)}{\Gamma(a+b)}-n\frac{\Gamma^{'}(b)}{\Gamma(b)}+\sum_{i=1}^{n}log(1-x_i)=0$$

Using the Newton-Raphson, as described below.

\begin{equation}
\hat{\theta}_{i+1}=\hat{\theta_i}-G^{-1}g,
\end{equation}

where $g$ is the vector of normal equations for wich
$$g=[g_1,g_2]$$
whit
$$g_1=n\psi(a+b)-n\psi\left(a\right)+\sum_{i=1}^{n}log(x_i)$$
$$g_2=n\psi(a+b)-n\psi\left(b\right)+\sum_{i=1}^{n}log(1-x_i)$$
$$G=\left[ \begin{array}{cc}
\frac{\partial g_1}{\partial a} & \frac{\partial g_1}{\partial b} \\
\frac{\partial g_2}{\partial a} & \frac{\partial g_2}{\partial b}
\end{array} \right]$$

$\psi(u)$ and   $\psi'(u)$ are the di- and tri-gamma functions defined as $\psi(u)=\frac{\Gamma^{'}(u)}{\Gamma(u)}$ and $\psi'(u)=\frac{\Gamma^{''}(u)}{\Gamma(u)}-\frac{\Gamma^{'}(u)^{2}}{\Gamma(u)^{2}}$.

\subsection{\textbf{Method of moments estimators}}
The method of moments estimators of the Beta-Kotz distribution parameters involve equating of the two moments of the Beta-Kotz distribution with the sample mean and the variance, see \cite{Farnun}.

The moment generating function for a moment of order $t$ is
$$E(X^t)=\int_{0}^{1}{\frac{\Gamma(a+b)}{\Gamma\left(a\right)\Gamma\left(b\right)}}x^{a-1}(1-x)^{b-1}dx$$
$$=\dfrac{\Gamma(a+t)\Gamma(a+b)}{\Gamma((a+b+t)\Gamma(a)}$$

The method of moments estimators are found by setting the sample mean, $\bar{X}$, the  variance, $S^2$, the equal to the population mean, variance.

$$\bar{X}=\frac{a}{a+b}$$

$$S^{2}=\frac{ab}{(a+b)^2(a+b+1)}$$

Then
\begin{equation}
\hat{a}=\bar{X}\left(\dfrac{\bar{X}(1-\bar{X})}{S^{2}}-1\right)
\end{equation}
\begin{equation}
\hat{b}=(1-\bar{X})\left(\dfrac{\bar{X}(1-\bar{X})}{S^{2}}-1\right)
\end{equation}

\section{Measuring The Credit Risk}

Daily, the financial institutions are exposed to different risk factors that threaten business continuity and the stability of the global financial system. Risk can be defined as the degree of uncertainty about future net returns. The literature distinguishes four main types of risk. credit risk is the risk of losses resulting from failure by its financial counterparties to meet their obligations. Market risk relates to uncertainty of future earnings, due to the changes in market conditions. Liquidity risk occurs when the entity is unable to adequately fulfil its obligations due to lack of liquid resources. Operational risk relates to the loss due to failures in the procedures that affect the operation of the entity. The most important of this risk is credit risk since this is the core of the business in financial institutions. It has three basic components: credit exposure, the probability of default and loss in the event of default. In this section, we apply the measures obtained in section \ref{Valuerisk}, namely: Value-at-Risk, Conditional Value-at-Risk and  Economic Capital over the loss in the event of default.

As in \cite{Ward(2002)}, in our model, we work with the credit loss rate R, which is the total credit loss C of the institution divided by the total exposure and we have chosen to use a Beta distribution to model the portfolio of loans.

In case of a portfolio with $N$ borrows, the portfolio loss variable $L=L_{N}$ is defined in \cite{Lutkebohmert(2014)} as

$$L_{N} = \sum_{n=1}^{N}EAD_n\cdot LGD_n\cdot PD_n$$
where

$L_N$: Loss is the amount that an institution is contractually owed but does not receive because of the borrower or borrower defaulting.

$EAD_n$: Exposure at default is the total amount of the institution's liability to a borrower.

$LGD_n$:Loss Given Default is the fraction of the exposure that is actually lost given a default of that borrower.

$PD_n$: Probability of Default is the binomially distributed Bernoulli random variable that measures whether a borrower has defaulted or not. It takes the values of either one in the case of default, or zero otherwise.

The expected credit loss is the average annual loss rate over the course of a business cycle.  Because the default is modelled as Bernoulli and it does not allow firms to default repeatedly without curing, the sum of a correlated portfolio of loans follows a Beta-Kotz distribution.

\subsection{Implementation}

We will consider the consumer portfolio model designed by the Superintendencia Financiera de Colombia (SFC), which is used in Colombia for the evaluation and supervision of internal models presented by financial institutions.

In the case of SFC there are 6 rating categories, the highest credit quality being AA, and the lowest, CC; the last state is the default. Default corresponds to the situation where an obligor cannot make a payment related to a loan obligation, whether it is a coupon or the redemption of principal.

For the financial institutions in study, only default risk is modeled, not downgrade risk. As in \cite{Crouhy et al. (2000)} it is assumed that:
\begin{itemize}
	\item for a loan, the probability of default in a given period, say 1 month, is the same for any other month.
	\item for a large number of obligors, the probability of default by any particular 	obligor is small, and the number of defaults that occur in any given period is independent of the number of defaults that occur in any other period.
\end{itemize}

The distribution of default losses for a portfolio is derived in three stages:

\begin{itemize}
	\item Probability of Default (PD): It corresponds to the probability that in a lapse of twelve (12) months the borrower of a certain segment and qualification of the consumer portfolio will default. The Probability of Default is chosen from a rating system (Table 7) proposed by the SFC.	 
	
	\item Loss given default (LGD): Is the loss in the event of default of an obligor, the counterparty incurs a loss equal to the amount owned by the obligor (the exposure, i.e., the marked-to-market value if positive, and zero otherwise, at the time of default) less a recovery amount. The Loss given default is chosen from the SFC (Table 8).
	
	\item Distribution of default losses for a portfolio: In order to derive the loss distribution for a well-diversified porfolio, the distribution of default losses is calculed as the exposure times probability of defualt times the loss given default rate.
\end{itemize}

\begin{center}
	\begin{table}[H]\label{PD}
		\caption{ Probability of Default from SFC }
		\medskip
		\hspace{1.2cm}
		\begin{tabular}{|c|c|c|c|c|c|}
			\hline
			Rating&  Automobiles&	Other& Credit card& CFC Automobiles&
			CFC Other  \\
			\hline
			AA	&0,97\% &	2,10\%	&1,58\%&1,02\% &	3,54\%  \\
			\hline
			A&	3,12\%&	3,88\%&	5,35\%&	2,88\%&	7,19\%
			\\
			\hline
			BB&	7,48\%&	12,68\%&	9,53\%&	12,34\%&	15,86\% \\
			\hline
			B &	15,76\%&	14,16\%&	14,17\%&	24,27\%&	31,18\%   \\
			\hline
			CC&	31,01\%&	22,57\%&	17,06\%&	43,32\%&	41,01\%  \\
			\hline
			Default &100.0\% &100.0\% &100.0\% &100.0\% &100.0\% \\
			\hline	
		\end{tabular}
	\end{table}
\end{center}

\begin{table}[H]\label{LGD}
	\caption{ Loss given default from SFC }
	\medskip
	\hspace{2cm}
	\begin{tabular}{|c|c|c|c|c|c|}
		\hline
		Type of Guarantee& PD &Days &PD &  Days &PD  \\
		\hline
		Admissible Guarantee  & & & & & \\
		\hline
		
		Admissible financial collateral&	0-12&	-&	-&	-&	-\\
		\hline
		Commercial and residential real estate&	40\%&	360&	70\%&	720&	100\% \\
		\hline
		Goods given in real estate leasing&	35\%&	360&	70\%&	720&	100\% \\
		\hline 	
		Goods given in leasing other than real estate&	45\%&	270&	70\%&	540&	100\% \\
		\hline
		Receivables&	45\%&	360&	80\%&	720&	100\%  \\    	\hline 	
		Other Guarantees suitable &	50\%&	270&	70\%&	540	& 100\% \\    	\hline
		No-admissible Guarantee &	60\%&	210&	70\%&	420& 100\%	 \\    	\hline
		No-Guarantee&	75\%&	30&	85\%&	90&	100\%     \\    	\hline
	\end{tabular}
\end{table}

The financial institutions in the study hold a portfolio of loans from almost 14.000 different obligors. In Table 9, we show the Rating, the exposure at default, the probability of default, the loss given default, the credit loss and the credit loss rate for some obligors.

\begin{center}
	\begin{table}[H]\label{Creditl}
		\caption{ Credit loss and Credit loss rate }
		\medskip
		\hspace{1.2cm}
		\begin{tabular}{|c|c|c|c|c|c|}
			\hline
			Rating&EAD  &PD &LGD  & Credit loss & Credit loss rate  \\
			\hline
			CC	& \$ 391,967& 	22.57\%&	60\%&	 \$ 53,080 &	0.00728\% \\
			\hline
			AA&	 \$ 9,725,044& 	2.10\%&	60\%&	 \$ 122,536& 	0.01681\%   \\
			\hline
			CC&	 \$ 1,327,760& 	22.57\%	&60\%&	\$ 179,805& 0.02466\% \\
			\hline
			CC	&  \$ 1,134,433 &22.57\%&	60\% &	\$ 153,625 & 	0.02107\%   \\
			\hline
			CC	& \$ 296,882 &	22.57\%&	60\% &	 \$ 40,204 	&0.00551\%     \\
			\hline
			CC&	 \$ 708,982 &	22.57\%	&60\% &	\$96,010 &	0.01317\%      \\
			\hline
			CC&	 \$ 71,606& 	22.57\%&	60\%&	 \$ 9,697 &	0.00133\%      \\
			\hline
			CC	& \$ 1,079,607 &	22.57\% &	60\%	 & \$ 146,200 	&0.02005\%   \\
			\hline
			CC&	 \$ 626,049& 	22.57\% &	60\% &	\$ 84,780 &	0.01163\%  \\
			\hline
			
			CC&	 \$ 1,781,217 &	22.57\%&	65\%	& \$ 261,313& 	0.03584\% \\
			\hline
			CC	& \$ 1,465,135 &	22.57\%	&60\%	& \$ 198,409 &	0.02721\%   \\
			\hline
			CC	& \$ 779,251 & 	22.57\%&	60\%	& \$ 105,526 &	0.01447\%  \\
			\hline
			
			B	& \$ 200,175& 	14.16\%&	65\%	& \$ 18,424&0.00253\%		\\
			\hline 												
			BB&	 \$ 297,777 &12.68\%&	60\% & \$ 22,655 &	0.00311\%		\\
			\hline 																					
			BB	 &\$ 342,253 	&12.68\%&	60\%	& \$ 26,039 & 0.00357\%		\\
			\hline 										
			CC	& \$ 139,452& 	22.57\%	&60\%	& \$ 18,885 &	0.00259\%	 	\\
			\hline 																		
			AA&	 \$ 314,744 &	2.10\%&	60\%	& \$ 3,966 &	0.00054\%		\\
			\hline 																
			...&...&...&...&...&...
			\\
			\hline
		\end{tabular}
	\end{table}
\end{center}

After calculating the loss rate  derived from the credit risk, the Beta-Kotz  distribution is fit using the method of moments estimators see \cite{Farnun}, where

\begin{equation}
\hat{a}=\bar{X}\left(\dfrac{\bar{X}(1-\bar{X})}{S^{2}}-1\right)
\end{equation}
\begin{equation}
\hat{b}=(1-\bar{X})\left(\dfrac{\bar{X}(1-\bar{X})}{S^{2}}-1\right)
\end{equation}

The values for the parameters of the Beta-Kotz distribution of each month analyzed are shown in Table 10.

\begin{table}[H]\label{fit}
	\caption{ Parameter estimation}
	\medskip
	\hspace{0.0001cm}
	\begin{tabular}{|c|c|c|c|c|c|c|c|c|c|c|c|c|c|}
		\hline
		& Jan.&Feb.&Mar.&Abr.&May. &Jun.&Jul.&Ago.&Sep.&Oct.&Nov.&Dic.
		\\
		\hline
		$a$&0.199	& 0.201& 0.190	&0.208&0.205& 0.186&0.191& 0.185&0.184&0.181& 0.190&0.178
		\\
		\hline
		$b$ & 30.63&30.81& 31.75	&31.97&30.96&28.63&29.40&28.38&28.27&27.12&26.76&25.36	\\ 		\hline
	\end{tabular}
\end{table}

Finally, Table 11 shows the risk measures calculated using the approaches developed in the Section (\ref{numeric}). On this table can be seen that an average for the 2017 year, the total Expected credit loss was \$ 720.495.087 with a maximum loss of  \$ 2.185.686.058 to 99\% of confidence, the economic capital needed as a buffer against unexpected losses  was \$1.465.190.972, and on average, higher losses could be reported in the amount of \$4.787.439.041.

\begin{center}
	\begin{table}[H]\label{Measures} 
		\caption{ Measures Credit Risk }
		\medskip
		\hspace{2cm}
		\begin{tabular}{|c|c|c|c|c|c|}
			\hline 
			Month &$E(X)$	&$VaR_{\alpha}(X)$ &$EC_{\alpha}(X)$ &$CVaR_{\alpha}(X)$  \\ 
			\hline
			Jan. & \$ 729.135.016
			&\$ 2.227.701.427& \$ 1.498.566.453&  \$ 4.836.908.886
			\\ 
			\hline
			
					Feb.& \$ 733.684.644 & \$ 2.241.519.936
			&\$   1.507.835.292 & \$ 4.859.629.904	
			
			\\ 
			\hline 
			Mar.& \$ 729.848.669
			&\$  2.047.888.348& \$ 1.318.039.679  &\$ 4.544.330.501

			\\ 
			\hline
			Abr.& \$737.507.092
			&\$ 2.251.881.449
			& \$ 1.514.374.358 	& \$ 4.849.257.893
			\\ 
			\hline
			
			May.&   \$  732.118.086
			&\$ 2.234.220.031
			& \$ 1.502.101.945  & \$ 4.793.813.506
			\\ 
			\hline 
			Jun.&  \$  734.764.136
			&	\$ 2.244.531.943
			& \$ 1.509.767.807   & \$ 4.918.483.683
			\\ 
			\hline 
			Jul.& \$ 727.331.803
			& \$ 2.222.288.246
			& \$ 1.494.956.443 	& \$ 4.863.703.380
			\\ 
			\hline 
			Ago.& \$ 729.925.682
			&\$ 2.229.425.636
			&\$ 1.499.499.954  	& \$ 4.926.333.695
			\\ 
			\hline 
			Sep.& \$  737.451.217
			&\$ 2.252.358.797
			&\$ 1.514.907.580
			&\$ 4.967.022.930   
			\\ 
			\hline 
			Oct.&\$ 748.696.016
			&\$ 2.285.068.627
			&\$  1.536.372.610&\$ 5.082.848.683
			\\ 
			\hline 
			Nov.&\$ 654.108.600
			& \$ 2.001.390.110
			&\$ 1.347.281.510  &\$ 4.417.603.102
			\\ 
			\hline 
			Dic.&\$  651.370.077
			& \$1.989.958.149
			& \$ 1.338.588.072   &\$  4.389.332.331
			\\ 
			\hline 
			
		\end{tabular} 
	\end{table}
\end{center}

\section{Appendix}

\textbf{A.1} \textbf{Proof of Proposition \ref{1} }
\begin{proof}  The joint density of $A$ and $B$ is
	\begin{equation*}
	f_{A,B}(A,B)=\frac{\pi^{\frac{n_{1}}{2}+\frac{n_{2}}{2}}}{\Gamma\left(\frac{n_{1}}{2}\right)\Gamma\left(\frac{n_{2}}{2}\right)
		\Sigma^{\frac{n_{1}}{2}+\frac{n_{2}}{2}}}A^{\frac{n_{1}}{2}-1}B^{\frac{n_{2}}{2}-1}
	h_{1}\left(\Sigma^{-1}A\right)h_{2}\left(\Sigma^{-1}B\right)dAdB.
	\end{equation*}
	Take $C=A+B$, $A=T^2X$ and $C=T^2$, where $T>0$, so $C-A=T^2(1-X)$.
	and $dAdB=T^2dXdT^2$.
	
	Then the joint density function of $T^2$ and $X$ is
	\begin{equation*}
	\frac{\pi^{\frac{n_{1}}{2}+\frac{n_{2}}{2}}x^{\frac{n_{1}}{2}-1}(1-x)^{\frac{n_{2}}{2}-1}(t^2)^{\frac{n_{1}}{2}+\frac{n_{2}}{2}-1}}{\Gamma\left(\frac{n_{1}}{2}\right)\Gamma\left(\frac{n_{2}}{2}\right)
		\Sigma^{\frac{n_{1}}{2}+\frac{n_{2}}{2}}}
	h_{1}\left(\Sigma^{-1}t^{2}x\right)h_{2}\left(\Sigma^{-1}t^{2}(1-x)\right)dt^2dx
	\end{equation*}
	Assuming that the elliptical generator $h_{2}(w)$ has a Taylor expansion, we obtain
	\begin{eqnarray*}
		f_{T^{2},X}(t^{2},x)&=&   \frac{\pi^{\frac{n_{1}}{2}+\frac{n_{2}}{2}}x^{\frac{n_{1}}{2}-1}(1-x)^{\frac{n_{2}}{2}-1}}
		{\Gamma\left(\frac{n_{1}}{2}\right)\Gamma\left(\frac{n_{2}}{2}\right)
			\Sigma^{\frac{n_{1}}{2}+\frac{n_{2}}{2}}}\\
		&&\times\sum_{k=0}^{\infty}\frac{h_{2}^{k}(0)\Sigma^{-k}}{k!(1-x)^{-k}}
		(t^2)^{\frac{n_{1}}{2}+\frac{n_{2}}{2}-1+k}
		h_{1}\left(\Sigma^{-1}t^{2}x\right)dt^2dx
	\end{eqnarray*}
	The required density follows by integration over $T^{2}>0$. This integral is not trivial and  must be recovered from the following general result over positive definite matrices given by \cite{Caro2010}:
	
	\begin{equation*}
	\int_{\mathbf{X}>\mathbf{0}} h(tr \mathbf{X}\mathbf{Z})|\mathbf{X}|^{a-(m+1)/2}(d\mathbf{X}) =\frac{\Gamma_{m}(a)}{\Gamma(ma)}|\mathbf{Z}|^{-a}
	\int_{0}^{\infty}h(w)w^{ma-1}dw.
	\end{equation*}

	Taking $m=1$, $a=\frac{n_{1}}{2}+\frac{n_{2}}{2}+k$ and $Z=\Sigma^{-1}X$, $X=T^{2}$, we have the required result.
\end{proof}

\textbf{A.2}  \textbf{Proof of Theorem \ref{t2}}
\begin{proof} By definition (\ref{VaR}), if $X\sim KBeta(n_1,n_2,t_1,t_2)$, then the $VaR\beta_{\alpha}(X))$ is a real number,  such that
	$$F_X(VaR\beta_{\alpha}(X))=P(X\leq VaR\beta_{\alpha}(X))=\alpha.$$
	
	Equivalently
	
	$$\dfrac{\Gamma\left(t_1+t_2+\frac{n_1}{2}+\frac{n_2}{2}-2\right)}{\Gamma\left(t_1+\frac{n_1}{2}-1\right)\Gamma\left(t_2+\frac{n_2}{2}-1\right)}\int_{0}^{VaR\beta_{\alpha}(X)}x^{\left(t_1+\frac{n_1}{2}-1\right)-1}(1-x)^{\left(t_2+\frac{n_2}{2}-1\right)-1}dx=\alpha.$$
	
	By the incomplete Beta function´s definition, if $ t_1+\frac{n_1}{2}\neq 0,-1,-2,...$ the previous equation is equivalent to,
	\begin{equation}\label{w}
	C\frac{VaR\beta_{\alpha}(X)^{t_1+\frac{n_1}{2}-1}}{(t_1+\frac{n_1}{2}-1)}{}_2F_{1}\left(t_1+\frac{n_1}{2}-1,-t_2-\frac{n_2}{2}+2;  t_1+\frac{n_1}{2},VaR\beta_{\alpha}(X)\right)-\alpha=0
	\end{equation}
\end{proof}

\textbf{A.3}	\textbf{Proof of Theorem \ref{t3}}
\begin{proof}
	Let $$f(VaR\beta_{\alpha}(X))=\dfrac{\Gamma(a+b)}{\Gamma(a)\Gamma(b)}\sum_{k=0}^{\infty}\dfrac{(1-b)_{k}}{(a+k)k!} VaR\beta_{\alpha}(X))^{a+k}$$
	
	Thus
	$$f'(VaR\beta_{\alpha}(X))=\dfrac{\Gamma\left(a+b\right)}{\Gamma\left(a\right)\Gamma\left(b\right)}\sum_{k=0}^{\infty}\dfrac{(1-b)_{k}}{k!} VaR\beta_{\alpha}(X)^{a+k-1}$$
	
	$$=\dfrac{\Gamma\left(a+b\right)}{\Gamma\left(a\right)\Gamma\left(b\right)}VaR\beta_{\alpha}(X)^{a-1} (1-VaR\beta_{\alpha}(X))^{b-1}$$
	
	And
	$$f''(VaR\beta_{\alpha}(X))=\dfrac{\Gamma\left(a+b\right)}{\Gamma\left(a\right)\Gamma\left(b\right)}VaR\beta_{\alpha}(X)^{a-2} (1-VaR\beta_{\alpha}(X))^{b-2}\left[(a-1)-VaR\beta_{\alpha}(X)\left(a+b-2\right)\right]$$
	
	In general, if $h(VaR\beta_{\alpha}(X))=f'(VaR\beta_{\alpha}(X))$, then  $h^{n}(VaR\beta_{\alpha}(X))$ for $n=0,1,...$  is given by
	
	\begin{eqnarray*}
		&&\dfrac{\Gamma\left(a+b\right)}{\Gamma\left(a\right)\Gamma\left(b\right)}VaR\beta_{\alpha}(X)^{a-1-n} (1-VaR\beta_{\alpha}(X))^{b-1-n}
		\\\times
		&&\sum_{k=0}^{n}\binom{n}{k}
		\left(b-k\right)_{k}\left(a-(n-k)\right)_{n-k} VaR\beta_{\alpha}(X)^{k}(1-VaR\beta_{\alpha}(X))^{n-k}
	\end{eqnarray*}
	
	Thus $VaR\beta_{\alpha}(X)=0$ and $VaR\beta_{\alpha}(X)=1$ are the unique critical points and $VaR\beta_{\alpha}(X)=0$, $VaR\beta_{\alpha}(X)=1$ and $VaR\beta_{\alpha}(X)=\dfrac{a-1}{a+b-2}$ are the inflection points of the function.
	
	Given that $f(VaR\beta_{\alpha}(X))$ is a continuous function over a closed bounded interval $[0,1]$ and $f(0)=1-\alpha<0$ and $f(1)=1-\alpha>0$, then by continuity, there is at least one root in $(0,1)$.  Suppose that $f(VaR\beta_{\alpha}(X))$  has two zeros in $(0,1)$, which we will call $x_1$ and $x_2$, by Rolle's theorem, there exists at least one point c belonging to the interval $(x_1,x_2)$ such that $f'(c) = 0$, ie $c$ is a critical point, which is a contradiction since $0$ and $1$ are the unique critical points in $(0,1)$.
	
\end{proof}

\textbf{A.4} \textbf{Proof of Proposition \ref{p1}}
\begin{proof} 
	\begin{enumerate}
		\item As $b=2$ then the equation (\ref{m}) is equivalent to
		\begin{equation}\label{a1}
		\left(a+1\right)VaR\beta_{\alpha}(X)^{a}-a VaR\beta_{\alpha}(X)^{a+1}-\alpha=0.
		\end{equation}	
		
		\begin{itemize}
			\item If $a=b-1$, the  polynomial in  (\ref{a1}) is reduced to
			$-VaR\beta_{\alpha}(X)^{2}+2VaR\beta_{\alpha}(X)-\alpha$ whose zeros are given by
			$VaR\beta_{\alpha}(X)=1\pm \sqrt{1-\alpha}$. As $0<\alpha<1$ then $1-\alpha>0$, thus $1- \sqrt{1-\alpha}\in (0,1).$\\
			
			\item If $a=b$ the  polynomial in (\ref{a1}) is reduced to $-2VaR\beta_{\alpha}(X)^{3}+3VaR\beta_{\alpha}(X)^{2}-\alpha$ whose real zero is given by $VaR\beta_{\alpha}(X)= \frac{1}{2}+\frac{1-\sqrt{-3}}{4(-1+2\alpha+2\sqrt{-a+a^2})^{\frac{1}{3}}} +\dfrac{1}{4}(1+\sqrt{-3})(-1+2\alpha+2\sqrt{-a+a^2}) ^{\frac{1}{3}}$.\\

			\item If $a=b+1$ the  polynomial (\ref{a1}) is reduced to  $-3VaR\beta_{\alpha}(X)^{4}+4VaR\beta_{\alpha}(X)^{3}-\alpha$ whose real zeros are given by 	
			$VaR\beta_{\alpha}(X)=\frac{1}{6}\left[2+\sqrt{4+\frac{6(\alpha+Q^{2})}{R}}\pm \sqrt{2}\sqrt{4-\frac{3\alpha}{Q}}\right]+\frac{1}{6}\left[-3Q+\frac{4\sqrt{2}}{\sqrt{2+\frac{3(\alpha+R^{2})}{R}}}\right].$ Where, $Q=(\alpha-\sqrt{-(-1+\alpha)\alpha^{2}})^{\frac{1}{3}}$ and $R=(\alpha-\sqrt{-(-1+\alpha)\alpha})^{\frac{1}{3}}$, whose real zero in $(0,1)$ is given by
			
			$$\frac{1}{6}\left[2+\sqrt{4+\frac{6(\alpha+Q^{2})}{R}}- \sqrt{2}\sqrt{4-\frac{3\alpha}{Q}}\right]+\frac{1}{6}\left[-3Q+\frac{4\sqrt{2}}{\sqrt{2+\frac{3(\alpha+R^{2})}{R}}}\right].$$
		\end{itemize}
		
		\item   As $b=3$ then (\ref{m}) is equivalent to
		\begin{equation}\label{g}
		\dfrac{1}{2}\left(a+1\right)aVaR\beta_{\alpha}(X)^{a+2}-\left(a+2\right)aVaR\beta_{\alpha}(X)^{a+1}+\frac{1}{2}\left(a+1\right)\left(a+2\right)VaR\beta_{\alpha}(X)^{a}-\alpha=0.
		\end{equation}
		
		\begin{itemize}
			\item If $a=b-2$, the  polynomial (\ref{g}) is reduced to
			$VaR\beta_{\alpha}(X)^{3}-3VaR\beta_{\alpha}(X)^{2}+3VaR\beta_{\alpha}(X)-\alpha$  whose real zero is given by $VaR\beta_{\alpha}(X)=1+\sqrt[3]{\alpha-1}.$
			
			\item If $a=b-1$, the  polynomial (\ref{g}) is reduced to
			$3VaR\beta_{\alpha}(X)^{4}-8VaR\beta_{\alpha}(X)^{3}+6VaR\beta_{\alpha}(X)^{2}-\alpha$ whose real zeros are given by
			
			{\small $$\dfrac{2}{3}\pm\dfrac{1}{2}\sqrt{\left[\dfrac{8}{9}+\dfrac{16}{27\sqrt{\dfrac{4}{9}+\dfrac{2}{3}SP+\dfrac{2(1-\alpha)}{3SP}}}-\dfrac{2}{3}SP-\dfrac{2(1-\alpha)}{3SP}\right]}-\dfrac{1}{2}\sqrt{\dfrac{4}{9}+\dfrac{2}{3}SP+\frac{2(1-\alpha)}{3SP}},$$}
			
			where $S=(-1+\sqrt{a})^{\frac{2}{3}}$ and $P=(1+\sqrt{a})^{\frac{1}{3}}$, then the $VaR\beta_{\alpha}(X)\in(0,1)$
			
			{\small $$\dfrac{2}{3}+\dfrac{1}{2}\sqrt{\left[\dfrac{8}{9}+\dfrac{16}{27\sqrt{\dfrac{4}{9}+\dfrac{2}{3}SP+\dfrac{2(1-\alpha)}{3SP}}}-\dfrac{2}{3}SP-\dfrac{2(1-\alpha)}{3SP}\right]}-\dfrac{1}{2}\sqrt{\dfrac{4}{9}+\dfrac{2}{3}SP+\frac{2(1-\alpha)}{3SP}}$$}
			
		\end{itemize}
		
		\item As $b=4$ and  $a=b-3$ then (\ref{m}) is equivalent to
		$$-VaR\beta_{\alpha}(X)^{4}+4VaR\beta_{\alpha}(X)^{3}-6VaR\beta_{\alpha}(X)^{2}+4VaR\beta_{\alpha}(X)-\alpha=0,$$
		whose real zeros are given by  $VaR\beta_{\alpha}(X)=1\pm\sqrt[4]{\alpha-1}.$
		
		As $0<\alpha<1$ then $1-\alpha>0$, thus $VaR\beta_{\alpha}(X)=1-\sqrt[4]{\alpha-1}\in(0,1).$	
		
		\item As $b=1$ then (\ref{m}) is equivalent to  $VaR\beta_{\alpha}(X)^{n+1}-{\alpha}=0$ then  $VaR\beta_{\alpha}(X)=\pm\sqrt[n+1]{\alpha}$, thus $VaR\beta_{\alpha}(X)=\sqrt[n+1]{\alpha}\in(0,1).$	
	\end{enumerate}	
\end{proof}

\section{Concluding Remarks}

This work defines and studies the so called Beta Kotz distribution, a flexible robust model based on a family of distributions including the Gaussian law. The univariate  analysis developed in the paper emerged from a matrix-variate theory based on elliptical models which are usually out of the discussions in the classical univariate statistical books. 

This paper derives the so called Beta Kotz distribution based on a general family of distributions including the classical Gaussian model. Some properties and estimation methods are also provided.

The paper also develops efficient computational procedures and analytical solutions for estimating  some risk measures based on loss distributions on the credit risk framework. Under parameter restrictions, some  analytical expressions can be found, in the other cases, we introduce a numerical algorithm which allows for the computation of this risk measures.

Our results includes the proof for the unique zero of a particular Gauss hypergeometric function then a number of consequences can be derived for computation of the risk measures in different scenarios.

The methods developed here may be applicable to others distribution functions whose risk measures are defined in terms of hypergeometric functions, such as the Gamma, Pearson, Chi-Square, Studen t, among others.

The philosophy of the method can be easily extended to higher dimensions in order to study the VaR in the matrix variate setting. This topic is under research.

\section*{Acknowledgments}
The authors wish to express their gratitude to Doctorate in Mathematics (Doctoral School of Mathematics, IT and Telecommunications, University of Toulouse, Toulouse, France) and to Doctorate in Modelling and Scientific Computing (Faculty of Basic Sciences, University of Medellin, Medellin, Colombia).


\begin{thebibliography}{9}
	
	\bibitem[Alexander  (2008)]{Alexander(2008)} C. Alexander (2008) Market risk analysis IV: value-at-risk models. Chichester: John Wiley \& Sons.
	
	\bibitem[Arias et al. (2016)]{Arias(2016)} M. A. Arias-Serna, M.E. Puerta-Y\'epes, C.E. Escalante-Coterio \& G. Arango-Ospina (2016) $ (Q,r) $ Model with $ CVaR_{\alpha} $ of costs minimization, Journal of Industrial \& Management Optimization {\bf 13} (1), 135-146.
	
	\bibitem[Bain and Engelhardt (2000)]{Bain(2000)} L. J. Bain  \& M. Engelhardt (2000) Introduction to Probability and Mathematical Statistics. United states of America: Duxbury Press.
	
\bibitem[Baterman (1953)]{Baterman(1953)} H. Baterman (1953) Higher trasendental functions. New York: McGraw-Hill Book Company, Vol.I-III. 

	\bibitem[Ball (2000)]{Ball(2000)}
	J. S. Ball (2000) Automatic computation of zeros of Bessel functions and other special functions, Sci. Comput. {\bf21} (4), 1458-1464.
	
	\bibitem[Caro-Lopera \textit{et al} (2010)]{Caro2010} F. J. Caro-Lopera, G. Gonz\'alez-Far\'ias \& J.A. D\'{\i}az-Garc\'{\i}a (2010) 	Noncentral elliptical configuration density. Journal of Multivariate Analysis {\bf101}, 32-43.
	
	\bibitem[Caro-Lopera et al.(2014)]{Caro-Lopera(2014)} F. J.Caro-Lopera,  G. Gonz\'alez-Far\'ias \&  N. Balakrishnan (2014) On generalized Wishart distributions-I: Likelihood ratio test for homogeneity of covariance matrices. Sankhya A, {\bf76} (2), 179-194.
	
	\bibitem[Caro-Lopera et al.(2014a)]{Caro-Lopera(2014a)} F. J. Caro-Lopera, G. Gonz\'alez-Far\'ias,  Balakrishnan, N. (2014)  On Generalized Wishart Distributions - II: Sphericity Test. Sankhya A, {\bf76} (2), 195-218.
	
	\bibitem[Caro-Lopera et al.(2016)]{Caro2016}  F. J. Caro-Lopera, G. Gonz\'alez-Far\'ias \&  N. Balakrishnan (2016) Matrix-variate distribution theory under elliptical models-4: Joint distribution of latent roots of covariance matrix and the largest and smallest latent roots. Journal of Multivariate Analysis {\bf145}, 224-235.
	
	\bibitem[Crouhy et al.(2000)]{Crouhy et al. (2000)} M. Crouhy, D. Galai  \& R. Mark (2000) A comparative analysis of current credit risk models. Journal of Banking and Finance, {\bf24}, (1), 59-117.
	

	\bibitem[D\'iaz-Garc\'ia and Guti\'errez (2010)]{Diaz(2010)} J. A. D\'iaz-Garc\'ia \& R. Guti\'errez-J\'aimez (2010)	Zonal polynomials: A note, Chil. J. Statist. {\bf2} 21-30.
	
	

\bibitem[D\'iaz-Garc\'ia and Guti\'errez (2008)]{[Diaz(2008)} J. A. D\'iaz-Garc\'ia \& R. Guti\'errez-J\'aimez (2008). Singular matrix variate beta distribution, Journal of Multivariate Analysis, {\bf 9} (4), 637-648.
	
	
	\bibitem[Dominici et al. (2013) ]{Dominici(2013)}
	D. Dominici, Johnston,  S. J.  and Jordaan, K.(2013) Real zeros of 2F1 hypergeometric polynomials, J. Comput. Appl. Math., 247, 152-161.
	
	
	\bibitem[Duren and Guillou (2001)]{Duren(2001)} P. L. Duren \& B. J. Guillou (2001) Asymptotic Properties of Zeros of Hypergeometric Polynomials, Journal of Approximation Theory  {\sf111} (2), 329-343.
	
	\bibitem[Embrechts et al. (2001)]{Embrechts(2001)} P. Embrechts, A. McNeil \& D. Straumann (2001) Correlation and dependency in risk management: Properties and pitfalls. Risk Management: Value at Risk and Beyond. Cambridge: Cambridge University Press. 
	
	\bibitem[Farnum and Stanton (1987)]{Farnun} N. Farnum and L. Stanton (1987) Some Results Concerning the Estimation of Beta Distribution Parameters in PERT. The Journal of the Operational Research Society, {\bf38} (3), 287-290. 
	
	\bibitem[Fox (2016)]{Fox(2016)} J. A. Fox (2016) R and the journal of statistical software. Journal of Statistical Software {\bf73} (2), 1-13 .
	
	\bibitem[Glasserman et al. (2002)]{Glasserman(2002)} P. Glasserman, P. Heidelberger \& P. Shahabuddin (2002) Portfolio value-at-risk with heavy tailed risk factors. Mathematical Finance, {\bf12} (3), 239-269.
	
	\bibitem[Gabrielsen Alexandros et. al (2015)]{Gabrielsen(2015)} A. Gabrielsen, A. Kirchiner, Z. Liu \& P. Zagagli (2015) Forecasting value-at-risk with time-varying variance, skewness and kurtosis in an exponential weighted moving average framework. Annals of Financial Economics {\bf 10} (01), 1-29.
	
	
	\bibitem[Gu{\'e}gan and Tarrant (2012)]{Guegan(2012)} D. Gu{\'e}gan, 
	W. Tarrant (2012) On the necessity of five risk measures. Annals of Finance {\bf8} (4) 533--552.
	
	\bibitem[Gupta et al.(2013)]{Gupta(2013)} A. K. Gupta,  T. Varga,  \& T. Bodnar (2013) Elliptically contoured models in statistics and portfolio theory. Second Edition ed. New York: Springer. 
	
	\bibitem[Han et.al (2014)]{Han(2014)} C. Han, W. Liu \& T. Chen (2014) VaR/CVaR estimation under stochastic volatility models. International Journal of Theoretical and Applied Finance {\bf17} (02), 1-35.
	
	
	\bibitem[Johnson and Kotz(1970)]{Kotz(1970)}
	N. L. Johnson  \& S. Kotz  (1970) Continuous univariate distributions, Boston: Houghton Mifflin.co.
	
	
	\bibitem[Jorion (2007)]{Jorion(2007)} P. Jorion (2007) Financial risk manager handbook. Fourth edition ed. New Jersey: John Wiley \& Sons
	
	\bibitem[Kasper (2017)]{Kasper (2017)} W. V. Kasper Welbers (2017) Text Analysis in R. Communication Methods and Measures  {/bf11} (4), 245-265.
	
	\bibitem[Li (1997)]{Li(1997)} R. Li. (1997) The expected values of invariant polynomials with matrix argument of elliptical distributions. Acta Mathematicae Applicatae Sinica {\bf13}, 64-70.
	
	\bibitem[Lin and Shen (2006)]{Lin(2006)} C. Lin \& Shen, S (2006) Can the student-t distribution provide accurate value at risk? The Journal of Risk Finance, {\bf7} (3), 292.
	
	\bibitem[Lutkebohmert and Matchie (2014)]{Lutkebohmert(2014)} E. Lutkebohmert \&  E. Matchie (2014) Value-at-Risk computations in stochastic volatility models using second-order weak approximation schemes. International Journal of Theoretical and Applied Finance, {\bf17} (01), 1-26.
	
	\bibitem[Manganelli and Engle (2004)]{Manganelli(2004)} S. Manganelli \& R. F. Engle (2004). CAViaR: Conditional autoregressive value at risk by regression quantiles. Journal of Business \& Economic Statistics.
	
	\bibitem[Marinelli et.al (2012) ]{Marinelli(2012)} C. Marinelli,  E. D\'addona, S. Rachev (2012) Multivariate heavy-tailed models for value-at-risk estimation. International Journal of Theoretical and Applied Finance {\bf15} (4), 1-32.
	
	\bibitem[McNeil and Embrechts (2015)]{McNeil(2015)} A. J. McNeil, R. Frey \& P. Embrechts (2015) Quantitative risk management: Concepts, techniques and tools. Princeton, New Jersey: Princeton university press.
	
	
	\bibitem[Muirhead (2005)]{Muirhead(2005)} R. J. Muirhead (2005) Aspects of multivariate statistical theory. Wiley Series in Probability and Mathematical Statistics. New York: John Wiley and Sons. 
	
	
	\bibitem[Rockafellar and Uryasev (2000)]{Rockafellar(2000)} R. T. Rockafellar \& S. Uryasev (2000) Optimization of conditional value-at-risk. Journal of Risk {\bf2}, 21- 41.
	
	\bibitem[Rockafellar and Uryasev (2002)]{Rockafellar(2002)}  R.T. Rockafellar, S. Uryasev (2002) Conditional value-at-risk for general loss distributions, Journal of Banking \& Finance {\bf26} (7), 1443-1471.
	
	
	\bibitem[Srivastava et al. (2011)]{Srivastava(2011)} H. Srivastava, J.  Zhou \& Z. Wang (2011) Asymptotic distribution of the zeros of certain classes of hypergeometric functions and polynomials . Mathematics of Computation {\bf80} (2), 1769-1784.
	
	
	\bibitem[Stavroyiannis (2012)]{Stavroyiannis(2012)}  S. Stavroyiannis,  I. Makris, V. Nikolaidis and L. Zarangas (2012) Econometric modeling and value-at-risk using the Pearson type-IV distribution, International Review of Financial Analysis {\bf22}, 10-17.
	
	\bibitem[Wagalath and Zubelli (2018)]{Wagalath (2018)} L. Wagalath  \& J. Zubelli (2018) A liquidation risk adjustment for value at risk and expected shortfall. International Journal of Theoretical and Applied Finance, in print.
	
	\bibitem[Wang(2005)]{Wang(2005)} J. Z. Wang (2005) A note on estimation in the four-parameter Beta distribution Maximum , Communications in Statistics - Simulation and Computation {\bf34}, 495-50.
	
	\bibitem[Ward and Lee (2002)]{Ward(2002)} L. Ward \& D. Lee (2002) Practical application of risk-adjusted return on capital framework, CAS Forum Summer 2002, Dynamic Financial Analysis Discussion Paper.
	
	
	\bibitem[Yamai and Yoshiba (2005)]{Yamai(2005)} Y. Yasuhiro, Y.  Toshinao (2005) Value-at-risk versus expected shortfall: A practical perspective, Journal of Banking \& Finance {\bf29} (4), 997-1015.
	
	
\end{thebibliography}
\end{document}